\documentclass[11pt,twoside]{amsart}
\usepackage{url}
\usepackage{amssymb,amsthm,amsfonts,amstext,amsmath}
\usepackage{newcent}       
\usepackage{helvet}         
\usepackage{courier}        
\usepackage{graphicx}
\usepackage{enumerate}\usepackage{mathrsfs}
\usepackage{XCharter}

\usepackage[colorlinks=true,linkcolor=blue,citecolor=magenta]{hyperref}
\usepackage{bbm}
\usepackage{dsfont}
\usepackage{a4wide}
\usepackage{geometry}
\usepackage{multicol}
\numberwithin{equation}{section}
\usepackage{color}
\usepackage{float}
\usepackage{mathrsfs}
\usepackage{comment}

\newtheorem{theorem}{Theorem}[section]
\newtheorem{lemma}[theorem]{Lemma}
\newtheorem{proposition}[theorem]{Proposition}

\newtheorem{remark}[theorem]{Remark}

\newcommand{\mc}[1]{{\mathcal #1}}
\newcommand{\mf}[1]{{\mathfrak #1}}

\newcommand{\bb}[1]{{\mathbb #1}}

\newcommand{\eps}{\varepsilon}

\newcommand{\scn}{\;{;}\;}

\DeclareMathOperator{\TV}{TV}
\DeclareMathOperator{\LS}{LS}

\let\oldtocsection=\tocsection
\let\oldtocsubsection=\tocsubsection
\let\oldtocsubsubsection=\tocsubsubsection
\renewcommand{\tocsection}[2]{\hspace{0em}\oldtocsection{#1}{#2}}
\renewcommand{\tocsubsection}[2]{\hspace{1em}\oldtocsubsection{#1}{#2}}
\renewcommand{\tocsubsubsection}[2]{\hspace{2em}\oldtocsubsubsection{#1}{#2}}
\DeclareRobustCommand{\SkipTocEntry}[5]{}

\usepackage{hyperref}
\hypersetup{
	colorlinks=true,
	linkcolor=blue,
	filecolor=magenta,      
	urlcolor=magenta,
}

\begin{document}
	
\title[Sharp convergence to equilibrium for SSEP with reservoirs]{Sharp convergence to equilibrium for the SSEP with reservoirs}

\author{P. Gon\c calves}
\address[P. Gon\c calves]{Center for Mathematical Analysis, Geometry and Dynamical Systems, Instituto Superior T\'{e}cnico, Universidade de Lisboa, 1049-001 Lisboa, Portugal}
\email{pgoncalves@tecnico.ulisboa.pt}
\urladdr{\url{https://patriciamath.wixsite.com/patricia}}

\author{M. Jara}
\address[M. Jara]{Instituto de Matemática Pura e Aplicada, Estrada Dona Castorina, 110, 22460-320
	Rio de Janeiro, Brasil}
\email{mjara@impa.br}
\urladdr{\url{http://w3.impa.br/~monets/index.html}}

\author{R. Marinho}
\address[R. Marinho]{Center for Mathematical Analysis, Geometry and Dynamical Systems, Instituto Superior T\'{e}cnico, Universidade de Lisboa, 1049-001 Lisboa, Portugal}
\email{rodrigo.marinho@tecnico.ulisboa.pt}
\urladdr{\url{https://marinhor.weebly.com}}

\author{O. Menezes}
\address[O. Menezes]{Mathematics Department, Purdue University, 150 N. University Street, West Lafayette, IN 47907-2067 United States}
\email{omenezes@purdue.edu}
\urladdr{\url{https://www.math.purdue.edu/~omenezes/}}

\begin{abstract}
We consider the symmetric simple exclusion process evolving on the path of length $n-1$ in contact with reservoirs of density $\rho \in (0,1)$ at the  boundary. We use Yau's relative entropy method to show that if the initial measure is associated with a profile $u_0:[0,1] \to (0,1)$, then at explicit times $t^n(b)$ that depend on $u_0$,   the distance to equilibrium, in total variation distance, converges, as $n \to \infty$, to a profile
$\mathcal G(\gamma e^{-b})$. The parameter $\gamma$ also depends on the initial profile $u_0$ and $\mathcal G(m)$ stands for the total variation distance $\|\mc N (m,1) - \mc N(0,1)\|_{\TV}$.
\end{abstract}

\maketitle
\section{Introduction}

The convergence to the stationary state of the law of a finite-state, irreducible Markov chain is a classical problem in the theory of Markov chains. It is well known that for a \emph{fixed} Markov chain, the speed of convergence is exponential, with a rate equal to the \emph{spectral gap} of the chain. Driven by connections with MCMC, in the past 30 years the interest shifted to chains with \emph{large} state spaces. In what is usually called the \emph{modern} theory of Markov chains, convergence to stationarity is studied from a different point of view. One fixes a threshold distance, and the question is how much time does the chain need to get within this threshold from the stationary measure. The asymptotic analysis now leaves the threshold fixed and studies the time as the \emph{size} of the state space grows. This approach, pioneered by Aldous and Diaconis \cite{aldous}, leads to a new phenomenon, called the \emph{cut-off phenomenon}; see \cite{yuval} for an introduction. First, one fixes a way to measure the distance to the stationary state. From a probabilistic point of view, the most natural way is to use the \emph{total variation distance}. When the size of the chain is taken into account, convergence to the stationary state happens in a rather abrupt way: for times of order the so-called \emph{mixing time}, the law of the chain remains far from the stationary state. 
Then, during a \emph{time window} of order smaller than the mixing time, the law of the chain gets infinitesimally close to the stationary state. It is only after this mechanism that the convergence enters the exponential regime. Therefore, from the point of view of a fixed threshold, exponential convergence happens \emph{after} the threshold is met.

In order to discuss the point of view described above, it is necessary to consider a \emph{family} of Markov chains, parametrized by a parameter $n$ that is related to the size of the state space. The idea is that the size of the chain grows with $n$, and our task is to perform an asymptotic analysis of the distance to the stationary state as $n \to \infty$. In this article, we consider the symmetric exclusion process in contact with stochastic reservoirs, which is a family of Markov chains defined on the state space $\Omega_n:= \{0,1\}^{\{1,\dots,n-1\}}$, used to model the dynamics of boundary \emph{driven diffusive systems}, see \cite{BerD-SGabJ-LLan}. The dynamics of this model can be described as follows. Particles move as symmetric random walks under the exclusion constraint stating that there is at most one particle per site. Particles can be absorbed at the sites $x=0,n$ and particles can be  injected to the system whenever  the sites $x=1,n-1$ are empty.
It has been proved in \cite{gantert} that the mixing time of this model is of order $\mc O(n^2 \log n)$; the lower bound obtained in \cite{gantert} is conjectured to be optimal.

In a series of papers \cite{lacoin, lacoin profile, lacoin circle}, the author studied the mixing properties of the exclusion process on the circle and on the interval, obtaining sharp bounds on the mixing time that show the cut-off phenomenon for those models. Let us explain the estimate obtained in \cite{lacoin profile}, which is sharper than the usual cut-off estimates. Let $D_n(t)$ denote the total variation distance to the stationary state at time $t$. There exists a non-degenerated function $\mc G(b)$ and a positive constant $\lambda$ such that
\[
\lim_{n \to \infty} D_n \big( \tfrac{1}{2\lambda} n^2 \log n + b n^2 \big) = \mc G(b)
\]
for every $b \in \bb R$. This result is called \emph{profile cut-off} in the literature and it is stronger than cut-off as defined in \cite{yuval}. The function $\mc G$ is called the \emph{cut-off profile} and in the case of the exclusion process it can be expressed in terms of the total variation distance between two Gaussians with the same variance and with different means. The constant $\lambda$ turns out to be the \emph{spectral gap} of the Laplacian operator on the circle. In that case we say that the exclusion has \emph{Gaussian cut-off}. In \cite{gantert}, the authors show that the mixing time of the exclusion process with \emph{one} reservoir is also given by $\frac{1}{2 \lambda} n^2 \log n$, where $\lambda$ is now the spectral gap of the Laplacian operator on the interval with Dirichlet boundary conditions at one extreme of the interval and Neumann boundary conditions at the other extreme.

Observe that the classical definition of mixing time takes as initial condition for the chain the \emph{worst possible} initial state. 
In this article, we adopt a somewhat different point of view. Our aim is to observe the dependence of the mixing time with respect to the initial state of the chain. We take as initial states the so-called \emph{profile measures}, which are product measures whose densities are associated with a macroscopic initial density profile. Denote by $\nu_0^n$ one of these measures, see \eqref{mono} for a rigorous definition. These initial measures are natural in the context of hydrodynamic limits of interacting particle systems, see \cite{BerD-SGabJ-LLan}, \cite{claudios}. Assume that the reservoir rates at the boundary of the interval $\{1,\dots,n-1\}$ are equal. In that case the stationary state is explicit and it is a product measure. Let $D_n(t; \nu)$ be the total variation distance between the law at time $t$ of the process with initial measure $\nu$. We will prove that there exist positive constants $\lambda , \gamma$ such that
\begin{equation}
\label{gorila}
D_n \big(\tfrac{1}{2\lambda} n^2 \log n  + \tfrac{1}{\lambda} b n^2; \nu_0^n \big) = \mc G(\gamma e^{-b}),
\end{equation}
where the function $\mc G$ is the total variation distance between two Gaussians of unit variance, see \eqref{azapa}. In this limit we have chosen the normalisations to stress the universality of the function $\mc G$. The constant $\lambda$ is equal to the eigenvalue of the Dirichlet Laplacian in the interval associated with the first non-zero Fourier mode of the macroscopic density profile and the constant $\gamma$ depends on the corresponding Fourier coefficient. 

In \cite{lacoin circle} it is observed that the exclusion process on the circle starting from a \emph{typical} configuration does not present cut-off, and moreover the mixing time in that case is of order $\mc O(n^2)$. Observe that the constant in front of the leading order of the mixing time can be arbitrarily small, depending on how many Fourier modes of the initial profile vanish.
Therefore, our main result can be seen as an interpolation between the behavior of the mixing time starting from the worst possible configuration and starting from a typical configuration under the stationary measure.

The identification of the profile cut-off is known to be difficult. The heuristic derivation of the expressions appearing in \eqref{gorila} is not difficult to understand. The so-called \emph{relaxation time} describes the time it takes for the chain to forget the initial state when starting from a typical state. The relaxation time is equal to the inverse of the spectral gap $\lambda$. The so-called \emph{log-Sobolev constant} $K$ is another constant that can be used to describe the behavior of a Markov chain. In general we have the bound $K \leq \frac{\lambda}{2}$. It can be proved that the mixing time is bounded above by $- K^{-1}\log \min_\eta \mu(\eta)$, where $\mu$ is the stationary measure. If one believes that these bounds are sharp, one arrives exactly to the picture above, on which the distance to the stationary state goes from $0$ to $1$ in a window of size $\mc O(n^2)$, positioned at a time of order $\mc O(n^2 \log n)$. Most of the proofs of the cut-off behavior include a \emph{burn-in phase} of order $\mc O(\lambda^{-1})$, which serves to randomize the current configuration of the chain. If one is interested in the profile cut-off, this burn-in phase can be at most of size $o(\lambda^{-1})$, and at that size we know that the chain can not thermalize all the observables of the system. Therefore, in order to identify the cut-off profile, a deeper knowledge of the chain is necessary. 

Our aim in this work is to introduce a novel approach to the mixing time problem, which we call the \emph{relative entropy method}. This approach is based upon Yau's relative entropy method \cite{yau} used in the theory of hydrodynamic limits of interacting particle systems, see \cite[Chapter 6]{claudios}. An earlier entropy method, described in \cite{diaconislog}, derives upper and lower bounds on the mixing time in terms of the log-Sobolev constant. The idea is to estimate the relative entropy between the law of the process and the stationary state and to transport these estimates into total variation estimates by means of Pinsker's inequality. When available, this method usually gives bounds on the mixing time of the right order, but with non-matching constants.

Theorem \ref{t5} below implies in particular that the log-Sobolev constant of the chain $\{\eta_t \scn t \geq 0\}$ is bounded above by $K_{\LS}\big(\tfrac{1}{2},\rho\big)$, which implies an upper bound for the mixing time of the form $K_{\LS}\big(\tfrac{1}{2},\rho\big)  \log n + C$ for some $C$ large enough (recall here that time has already been scaled in a diffusive way). Wilson's method \cite{wilson} predicts that the mixing time is of order $\frac{1}{2\pi^2} \log n$. The constant $\frac{1}{2 \pi^2}$ is just the spectral gap of the Laplacian on the interval $[0,1]$ with Dirichlet boundary conditions. Observe that at the beginning of the evolution, the law of the chain is far from the invariant measure. Therefore, the relative entropy of the process with respect to the invariant measure is large, and the estimate using the log-Sobolev constant is not optimal. Our idea is to compare the law of the exclusion process at time $tn^2$ with the product measure $\nu_t^n$ associated to the solution of the hydrodynamic equation at time $t$. For these measures, Theorem \ref{t5} below provides a log-Sobolev inequality. If we assume that the law of the exclusion process is close to this product measure, then the log-Sobolev inequality would be more efficient. This strategy can be implemented as in \cite{mainlemma}, and it will give a sharp estimate for the relative entropy between the law of the process and this family of product measures. After that, the computation of the distance $D_n(t; \nu_0^n)$ is reduced to the computation of $\|\nu_t^n - \bar{\nu}_{\hspace{-1.5pt}\rho}^{\hspace{0.3pt}n}\|_{\TV}$. This last distance is simpler to compute, and full asymptotics can be readily obtained.

Whenever available, an effective way to prove the log-Sobolev is by means of \emph{comparison} of Dirichlet forms. Our proof of the log-Sobolev inequatily is based on a simple, albeit powerful observation: one can combine the Glauber dynamics at the reservoirs with the exclusion dynamics at the interior of the interval in order to compare the exclusion process with reservoirs with a Glauber dynamics that acts on each site of the interval. This comparison principle holds true even if the reservoirs are \emph{slow} \cite{baldasso, slow boundary}, see Theorem \ref{t5}. It is interesting to observe that our diffusive bounds on the log-Sobolev constant of the exclusion process in contact with reservoirs hold true exactly up to the same slow scale on which the reservoirs modify the hydrodynamic behavior of the system \cite{slow}.

The outline of this article is as follows. 
In Section \ref{s2}  we introduce the notation and we define the model under investigation. In Section \ref{s3} we present the proof of our main theorem, namely Theorem \ref{t1}. Section \ref{s4} is devoted to the proof of the log-Sobolev inequality which is our main tool to prove Theorem \ref{t1}. Section \ref{s5} is devoted to estimating the entropy production by using Yau's relative entropy method together with the log-Sobolev estimate. Finally, in Section \ref{s6} we compute the total variation distance between profile measures.   In the appendix we collected some results that were used along the proofs.

\section{Notation and main results}
\label{s2}
Let $n \in \{2,3,\dots,\}$ be a scaling parameter. Let $\Lambda_n:= \{1,\dots,n-1\}$ be the discrete interval with $n-1$ points. We will call the set $\{1,n-1\}$ the \emph{boundary} of $\Lambda_n$. We give to $\Lambda_n$ a graph structure by taking $E_n:= \{ \{x,x+1\} \scn x \in \{1,\dots,n-2\}\}$ as the set of edges in $\Lambda_n$. We call the vertex set $\Lambda_n$ the \emph{bulk} and we say that $x,y \in \Lambda_n$ are \emph{neighbors} if $\{x,y\} \in E_n$. In that case we write $x \sim y$. 

The set $\Omega_n := \{0,1\}^{\Lambda_n}$ is the state space of the Markov process described in the introduction, which we will rigorously define below. The elements $\eta = \{\eta(x) \scn x \in \Lambda_n\}$ of $\Omega_n$ are called \emph{configurations of particles}. We say that a vertex $x \in \Lambda_n$ is occupied by a particle (resp.~empty) in configuration $\eta \in \Omega_n$ if $\eta(x) =1$ (resp.~$\eta(x)=0$). 

Given a configuration $\eta \in \Omega_n$ and two vertices $x,y \in \Lambda_n$, we denote by $\eta^{x,y}$ the configuration of particles obtained from $\eta$ by exchanging the positions of the particles at $x$ and $y$, that is,
\[
\eta^{x,y}(z) = 
\left\{
\begin{array}{c@{\;\text{ if }\;}l}
\eta(x) & z =y,\\
\eta(y) & z=x,\\
\eta(z) & z \neq x,y.
\end{array}
\right.
\]
Given a configuration $\eta \in \Omega_n$ and a vertex $x \in \Lambda_n$, we denote by $\eta^{x}$ the configuration of particles obtained from $\eta$ by changing the value of $\eta(x)$ to $1-\eta(x)$, that is,
\[
\eta^{x}(z) = 
\left\{
\begin{array}{c@{\;\text{ if }\;}l}
1-\eta(x) & z =x,\\
\eta(z) & z \neq x.
\end{array}
\right.
\]
For $f: \Omega_n \to \bb R$ and $x,y \in \Lambda_n$, let $\nabla_{x,y} f, \nabla_x f: \Omega_n \to \bb R$ be defined as
\begin{equation}
\label{arica}
\nabla_{x,y} f(\eta) = f(\eta^{x,y})-f(\eta), \quad \nabla_x f(\eta) = f(\eta^x) -f(\eta)
\end{equation}
for any $\eta \in \Omega_n$. 

Let us fix a density $\rho \in (0,1)$. The symmetric simple exclusion process (SSEP) with reservoir density $\rho$ is the continuous-time Markov chain $\{\eta_t \scn t \geq 0\}$ with state space $\Omega_n$ and generated by the operator $\mf L_n$ given by
\[
\mf L_n f(\eta) := n^2 \sum_{x=1}^{n-2} \nabla_{x,x+1} f(\eta) + n^2 \!\!\!\!\!\! \sum_{x \in \{1,n-1\}} \!\!\!\!\! \big( \rho (1-\eta(x))+(1-\rho)\eta(x) \big) \nabla_x f(\eta)
\]
for any function $f: \Omega_n \to \bb R$ and any $\eta \in \Omega_n$. 
The factor $n^2$ speeds up time so that the process is observed in a diffusive time scale. 
Observe that the process $\{\eta_t \scn t \geq 0\}$ depends on $n$ and $\rho$. In order to simplify  notation, we do not make this dependence explicit in the notation. The same observation applies to the dependence in $\rho$ of $\mf L_n$, as well as to various other objects we define below.

The dynamics generated by $\mf L_n$ can be informally described as follows: in the bulk, particles perform nearest-neighbor random walks with rate $n^2$ under the \emph{exclusion} rule which forbids more than one particle at any vertex and at any time. At the boundary, Glauber dynamics inject (resp.~annihilate) particles independently at each empty (resp.~occupied) vertex in $\{1,n-1\}$ with rate $\rho n^2$ (resp.~$(1-\rho)n^2$). Here the  factor $n^2$ means that the Glauber dynamics  has the same intensity as the exchange dynamics acting on the whole system. 

For each function $u: \Lambda_n \to [0,1]$, let $\nu_{u(\cdot)}^n$ the Bernoulli product measure in $\Omega_n$ with density $u(\cdot)$, that is,
\[
\nu_{u(\cdot)}^n(\eta) := \prod_{x \in \Lambda_n} \big\{ \eta(x) u(x) + (1-\eta(x)) (1-u(x))\big\}
\]
for any $\eta \in \Omega_n$. Observe that when the values of $u$ belong to the \emph{open} interval $(0,1)$, the measures $\nu_{u(\cdot)}^n$ have full support. 

Since the process $\{\eta_t \scn t \geq 0\}$ is irreducible and has a finite state space, then it has a unique invariant measure, which turns out to be the  Bernoulli product measure $\bar{\nu}_{\hspace{-1.5pt}\rho}^{\hspace{0.3pt}n}$ associated with the constant function equal to $\rho$.

For each function $u: [0,1] \to [0,1]$, let $u^n: \Lambda_n \to [0,1]$ be defined by $u^n(x) := u\big( \tfrac{x}{n}\big)$ for every $x \in \Lambda_n$. 
For $\eps_0 \in (0,\min\{\rho,1-\rho\}]$ and $\kappa \geq 0$, let $\mc U_{\eps_0,\kappa}$ be the family of differentiable functions $u: [0,1] \to [0,1]$ such that:
\begin{itemize}

\item $\eps_0 \leq u(x) \leq 1-\eps_0$ for every $x \in [0,1]$;

\item $u(0) = u(1) = \rho$;

\item $|u'(x)| \leq \kappa$ for every $x \in [0,1]$.

\end{itemize}

Fix $\eps_0 \in (0,\min\{\rho,1-\rho\}]$ and $\kappa >0$ and let $u_0 \in \mc U_{\eps_0,\kappa}$. We call $u_0$ a \emph{profile} and we call the measures $\{\nu_{u^n_0(\cdot)}^n \scn n \in \{2,3,\dots\}\}$ the \emph{profile measures}. From now on, we consider the process $\{\eta_t \scn t \geq 0\}$ with initial distribution $\nu_{u^n_0(\cdot)}^n$. In order to simplify the notation, we define 
\begin{equation}
\label{mono}
\nu^n_0:= \nu_{u_0^n(\cdot)}^n.
\end{equation}

Let $D([0,\infty), \Omega_n)$ be the space of c\`adl\`ag trajectories in $\Omega_n$. We denote by $\bb P_{\nu^n_0}$ the probability measure in $D([0,\infty), \Omega_n)$ induced by the Markov process $\{\eta_t \scn t \geq 0\}$ with initial measure $\nu^n_0$. We denote by $\bb E_{\nu^n_0}$ the expectation with respect to $\bb P_{\nu^n_0}$. We denote the law of $\eta_t$ with respect to $\bb P_{\nu^n_0}$ by $\mu_t^n$.

The distance to equilibrium of the process $\{\eta_t \scn t \geq 0\}$ with initial measure $\nu^n_0$ is defined as 
\[
D_n(t; \nu^n_0):= \|\mu_t^n - \bar{\nu}_{\hspace{-1.5pt}\rho}^{\hspace{0.3pt}n}\|_{\TV},
\]
where $\|\mu -\nu \|_{\TV}$ stands for the total variation distance between the probability measures $\mu$ and $\nu$ in $\Omega_n$. 

In order to state our main result, we need to introduce the Fourier coefficients of the initial profile $u_0-\rho$. For $u_0: [0,1] \to [0,1]$ and $\ell \in \bb N$, define
\[
c_\ell(u_0) := \sqrt{2} \int (u_0(x)-\rho) \sin(\pi  \ell x) dx.
\]
We also need to define the \emph{Gaussian profile} $\mc G: \bb R \to [0,1]$ as
\begin{equation}
\label{azapa}
\mc G(m) := \|\mc N(m,1) - \mc N(0,1)\|_{\TV}= \tfrac{1}{2} \bb E \big[ \big| e^{m X -\tfrac{m^2}{2}}-1\big|\big],
\end{equation}
where $X$ is a random variable distributed as $\mc N(0,1)$.
Our main result is the following:

\begin{theorem}
\label{t1} Let $u_0:[0,1] \to [0,1]$ be differentiable. Assume that $u_0(0) = u_0(1) = \rho$ and that $u_0(x) \in (0,1)$ for every $x \in [0,1]$. Let $\ell_0 \in \bb N$ be the smallest integer such that $c_{\ell_0}(u_0) \neq 0$. For every $b \in \bb R$,
\[
\lim_{n \to \infty} D_n\big(t^n(b); \nu_0^n\big) = \mc G(\gamma e^{-b}),
\]
where 
\[
t^n(b):=  \frac{1}{2 \pi^2 \ell_0^2} \log n + \frac{b}{\pi^2 \ell_0^2} 
\text{ and } \gamma := \frac{|c_{\ell_0}(u_0)|}{\sqrt{ \rho(1-\rho)}}.
\]
\end{theorem}

\begin{remark}
Under the conditions on the profile $u_0$ stated on this theorem, there exist $\eps_0>0$ and $\kappa$ finite such that $u_0 \in \mc U_{\eps_0,\kappa}$.
\end{remark}

\section{Proof of Theorem \ref{t1}}\label{s3}

In this section we explain the strategy  we used to prove Theorem \ref{t1}. We start recalling the definition of relative entropy and Pinsker's inequality. Let $\nu$ be a probability measure in $\Omega_n$ and let $f$ be a density with respect to $\nu$. The \emph{relative entropy} of $f$ with respect to $\nu$ is defined as
\[
H_\nu (f) := \int f \log f d \nu.
\]
Relative entropy and total variation are related by Pinsker's inequality:

\begin{proposition}(Pinsker's inequality)
\label{p1}
Let $\mu$ and $\nu$ be two probability measures in $\Omega_n$. Let $f$ be the Radon-Nikodym derivative of $\mu$ with respect to $\nu$. Then $2\|\mu - \nu\|_{\TV}^2 \leq H_\nu(f)$.
\end{proposition}

We say that a probability measure $\nu$ in $\Omega_n$ is a \emph{reference measure} if $\nu(\eta) > 0$ for every $\eta \in \Omega_n$. Let $\nu$ be a reference measure, which may or may not depend on time, and let $f_t^n$  be the Radon-Nikodym derivative of $\mu_t^n$ with respect to $\nu$. By the triangle's inequality,
\begin{equation}
\label{iquique}
\big| D_n(t; \nu^n_0) - \|\nu - \bar{\nu}_{\hspace{-1.5pt}\rho}^{\hspace{0.3pt}n}\|_{\TV} \big| \leq \|\mu_t^n - \nu\|_{\TV} \leq \sqrt{\frac{H_\nu(f_t^n)}{2}}.
\end{equation}

Therefore, if we are able to find some reference measures $\nu^n_t$ for which $H_{\nu^n_t}(f_t^n)^{1/2}$ converges to $0$ faster than $D_n(t; \nu^n_0)$, then the proof of Theorem \ref{t1} is reduced to the computation of the distance $\|\nu^n_t - \bar{\nu}_{\hspace{-1.5pt}\rho}^{\hspace{0.3pt}n}\|_{\TV}$. 

Now we explain our choice for the reference measure $\nu^n_t$. Recall that we fixed constants $\eps_0 \in (0,\min\{\rho,1-\rho\}]$ and $\kappa >0$ and a profile $u_0 \in \mc U_{\eps_0,\kappa}$. For each $t \geq 0$ let us define $u_t^n: \{0,1,\dots,n\} \to [0,1]$ as
\[
u_t^n(x) :=
\left\{
\begin{array}{c@{\;\text{ if }\;}l}
\bb E_{\nu^n_0}[\eta_t(x)] & x \in \Lambda_n,\\
\rho & x \in \{0,n\}.
\end{array}
\right.
\]
Using Dynkin's formula,  it can be shown that $\{u_t^n \scn t \geq 0\}$ is the unique solution of the boundary-value problem
\begin{equation}
\label{antofagasta}
\left\{
\begin{array}{r@{\;=\;}l@{\;\text{ for }\;} l}
\tfrac{d}{dt} u_t^n(x) & \Delta_n u_t^n(x) & t \geq 0 \text{ and } x \in \Lambda_n,\\
u_t^n(x) & \rho & t \geq 0 \text{ and } x \in \{0,n\}, \\
u_0^n(x) & u_0\big(\tfrac{x}{n}\big) & x \in \Lambda_n. 
\end{array} 
\right.
\end{equation}
Here $\Delta_n$ is the discrete Laplacian operator defined on functions $f: \{0,1,\dots,n\} \to \bb R$ as
\[
\Delta_n f(x) = n^2 \big(f(x+1)+f(x-1)-2f(x)\big) 
\]
for any $x \in \Lambda_n$. 

Define $\nu^n_t : = \nu_{u_t^n(\cdot)}$. The property known as \emph{conservation of local equilibrium} \cite[Chp.~9]{claudios} states that for any $t \geq 0$ fixed, the measures $\mu_t^n$ and $\nu_t^n$ are close as $n \to \infty$ if observed on an interval on $\Lambda_n$ of fixed size. Therefore, it is reasonable to use the measures $\{\nu_t^n \scn t \geq 0\}$ as the reference measures to be plugged in \eqref{iquique}. Observe, however, that this conservation of local equilibrium is too weak to be useful in our situation: first, it only holds over a \emph{finite} time interval (we need to go to times of order $\mc O(\log n)$), and second it only holds on a \emph{finite} spatial interval (we need to go to the whole interval $\Lambda_n$). In Section \ref{s5} we will prove the following bound on the aforementioned relative entropy:

\begin{theorem}
\label{t2}
Let $\eps_0 \in (0,\min\{\rho,1-\rho\}]$ and $\kappa \geq 0$ be given. Let $u_0 \in \mc U_{\eps_0,\kappa}$ and let $f_t^n$ be the Radon-Nikodym derivative of the measure $\mu_t^n$ with respect to $\nu_t^n$. Define $H_n(t) := H_{\nu_t^n}(f_t)$. There exist constants $C_0=C_0(\eps_0,\kappa)$, $\delta_0=\delta_0(\eps_0,\kappa) >0$ such that
\[
H_n(t) \leq C_0 e^{-\delta_0 t}
\]
for every $n \in \{2,3,\dots\}$, every $u_0 \in \mc U_{\eps_0,\kappa}$ and every $t \geq 0$.
\end{theorem}

A version of this estimate was obtained in  \cite{mainlemma} in the context of non-equilibrium fluctuations from the hydrodynamic limit. Our novelty is the exponential decay as a function of $t$, which in particular allows to use this estimate over \emph{divergent} time windows.

The following theorem identifies the time window at which the convergence in total variation of $\nu_t^n$ to $\bar \nu_\rho^n$ happens. Recall the definition of the function $\mc G$ given in \eqref{azapa} and recall the definition of $\gamma$ given in Theorem \ref{t1}.
In Section \ref{s6} we prove the following result:

\begin{theorem}
\label{t3}
Under the conditions of Theorem \ref{t1}, for every $b \in \bb R$,
\[
\lim_{n \to \infty} \| \nu^n_{t^n(b)} - \bar{\nu}_{\hspace{-1.5pt}\rho}^{\hspace{0.3pt}n} \|_{\TV}
		= \mc G(\gamma e^{-b}).
\]
\end{theorem}

The main ingredient of the proof of Theorem \ref{t2} is a logarithmic Sobolev inequality for inhomogeneous product measures. The proof of this log-Sobolev inequality can be found in Section \ref{s4}.
Let $\Gamma_n$ be the \emph{carr\'e du champ} operator associated with $\mf{L}_n$: for every $f: \Omega_n \to \bb R$, $\Gamma_n f := \mf{L}_n f^2 - 2 f \mf{L}_n f$. 

\begin{theorem}
\label{t5} 
Let $\rho \in (0,1)$, $\eps_0 \in (0,\min\{\rho,1-\rho\}]$ and $\kappa >0$ be fixed. There exists a positive constant $K_0=K_0(\rho,\eps_0,\kappa)$ such that 
\begin{equation}
\label{mejillones}
H_{\nu_{u(\cdot)}^n}(f) \leq \frac{1}{K_0} \int \Gamma_n \sqrt{f} d \nu_{u(\cdot)}^n
\end{equation}
for every $u \in \mc U_{\eps_0,\kappa}^n$ and every density $f$ with respect to $\nu_{u(\cdot)}^n$.
\end{theorem}

\begin{proof}[Proof of Theorem \ref{t1}]
Fix $B >0$ and take $b \in [-B,B]$. By Theorem \ref{t2} and equation \eqref{iquique},
\[
\big| D_n(t^n(b); \nu^n_0) - \|\nu_{t^n(b)}^n - \bar{\nu}_{\hspace{-1.5pt}\rho}^{\hspace{0.3pt}n}\|_{\TV} \big| \leq C_1 e^{-\delta_0 t^n(b)/2},
\]
where $C_1= C_1(\eps_0,\kappa)$. Therefore,
\[
\big| D_n(t^n(b); \nu^n_0) - \|\nu_{t^n(b)}^n - \bar{\nu}_{\hspace{-1.5pt}\rho}^{\hspace{0.3pt}n}\|_{\TV} \big|\leq \frac{C_1 e^{2 \delta_1 B}}{n^{\delta_1}},
\]
where $\delta_1 = \frac{\delta_0}{4 \pi^2 \ell_0^2}$. The exact value of all these constants is not important; we only need that the right-hand side of this estimate converges to $0$ uniformly in $b \in [-B,B]$. From Theorem \ref{t3}, we conclude that
\[
\lim_{n \to \infty} D(t^n(b); \nu_0^n) = \mc G(\gamma e^{-b}),
\]
as we wanted to show.
\end{proof}

\section{The log-Sobolev inequality}
\label{s4}
In this section we prove the log-Sobolev inequality stated in Theorem \ref{t5}. In order to make the proof easier to follow, we prove a different version of the log-Sobolev inequality, from which Theorem \ref{t5} follows by comparison. 

For each $\kappa >0$, $n \in \{2,3,\dots\}$ and $\eps_0 \in (0,\min\{\rho,1-\rho\}]$, let $\mc U_{\kappa,\eps_0}^n$ be the class of functions $u: \Lambda_n \to [0,1]$ such that:
\begin{itemize}

\item $\eps_0 \leq u(x) \leq 1-\eps_0$ for every $x \in \Lambda_n$;

\item $n|u(x+1) -u(x)| \leq \kappa$ for every $x \in \{1,\dots, n-2\}$.

\end{itemize}

Fix $\theta >0$. Recall \eqref{arica}. For each $f: \Omega_n \to \bb R$ and each $x \in \{1,\dots,n-2\}$, let us define
\[
\mc D_x(f) := \int \big( \nabla_x f(\eta)\big)^2 d \nu_{u(\cdot)}^n,
\]
\[
\mc D_{x,x+1}(f) := \int \big( \nabla_{x,x+1} f(\eta)\big)^2 d \nu_{u(\cdot)}^n,
\]
\begin{equation}
\label{tocopilla}
\mc D(f) := \frac{\theta}{n} \mc D_1(f) + \sum_{x=1}^{n-2} \mc D_{x,x+1}(f).
\end{equation}
The quadratic form $\mc D(f)$ turns out to be the Dirichlet form of a simple exclusion process in $\Omega_n$ with inhomogeneous drift and with a reservoir at $x=1$ with intensity $\frac\theta n$.  This interpretation will not be used in this article, and therefore we will not enter into more details.

We prove the following:

\begin{theorem}
\label{t6}
Let $\theta, \kappa >0$ and $\eps_0 \in (0,\tfrac{1}{2}]$. There exists a positive constant $K= K(\eps_0, \kappa, \theta)$ such that
\begin{equation}
\label{calama}
H_{\nu_{u(\cdot)}^n}(f) \leq \frac{n^2}{K} \mc D (\sqrt{f} ) 
\end{equation}
for every $u \in \mc U_{\kappa, \eps_0}^n$ and every density $f$ with respect to $\nu_{u(\cdot)}^n$.
\end{theorem}

Estimates of this kind are known in the literature as \emph{log-Sobolev inequalities}. The log-Sobolev constant $K_{\LS}$ is defined as the largest constant $K$ that satisfies \eqref{calama}. Theorem \ref{t6} shows that $K_{\LS}^{-1}$ is uniformly bounded in $n$.

We need the following result to prove Theorem \ref{t6}:

\begin{lemma}[Comparison of quadratic forms]
\label{l1}
There exists a finite constant $C=C(\eps_0,\kappa,\theta)$ such that for every $n \in \{2,3,\dots\}$, every $\ell \in \{2, \dots, n-1\}$, every $f: \Omega_n \to \bb R$ and every $u \in \mc U_{\kappa, \eps_0}^n$,
\[
\mc D_{\ell}(f) \leq C n \mc D(f).
\]
\end{lemma}

\begin{proof}
Observe that for each $x \in \{2,\dots,n-1\}$, 
\[
\nabla_x f(\eta) = \nabla_{x-1,x}f(\eta) + \nabla_{x-1} f(\eta^{x-1,x}) + \nabla_{x-1,x} f\big( (\eta^{x-1,x})^{x-1}).
\]
Using the inequality
\[
(a+b+c)^2 \leq 2(1+\beta)(a^2+b^2) + \big(1+\tfrac{1}{\beta}\big) c^2,
\]
valid for every $a,b,c \in \bb R$ and any $\beta >0$, we see that
\[
\begin{split}
\mc D_x(f) 
	&\leq 2(1+\beta) \int \Big ( \big(\nabla_{x-1,x}(f)\big)^2 + \big( \nabla_{x-1,x} f \big( (\eta^{x-1,x})^{x-1}\big)\big)^2\Big) d \nu_{u(\cdot)}^n\\
	&\quad + \big(1+\tfrac{1}{\beta}\big) \int \big( \nabla_{x-1} f(\eta^{x-1,x})\big)^2 d \nu_{u(\cdot)}^n
\end{split}
\]
for every $x \in \{2,\dots,n-1\}$.
Performing some changes of variables, we see that
\[
\int \big( \nabla_{x-1,x} f \big( (\eta^{x-1,x})^{x-1}\big)\big)^2 d\nu_{u(\cdot)}^n
		=\int \big( \nabla_{x-1,x}f(\eta)\big)^2 \frac{\nu_{u(\cdot)}^n\big( (\eta^{x-1})^{x-1,x}\big)}{\nu_{u(\cdot)}^n(\eta)} d\nu_{u(\cdot)}^n,
\]
\[
\int \big( \nabla_{x-1} f(\eta^{x-1,x})\big)^2 d\nu_{u(\cdot)}^n = \int \big( \nabla_{x-1} f(\eta)\big)^2 \frac{\nu_{u(\cdot)}^n(\eta^{x-1,x})}{\nu_{u(\cdot)}^n(\eta)} d\nu_{u(\cdot)}^n.
\]
Analyzing the four possible cases we see that
\[
\frac{\nu_{u(\cdot)}^n\big( (\eta^{x-1})^{x-1,x}\big)}{\nu_{u(\cdot)}^n(\eta)} \leq \frac{1}{\eps_0} -1,
\]
\[
\frac{\nu_{u(\cdot)}^n(\eta^{x-1,x})}{\nu_{u(\cdot)}^n(\eta)} \leq 1+\alpha,
\]
where $\alpha =\frac{\kappa}{ \eps_0^2 n} $.
We conclude that
\begin{equation}
\label{taltal}
\mc D_x(f) \leq \frac{2(1+\beta)}{\eps_0} \mc D_{x-1,x}(f) +\Big( 1+ \frac{1}{\beta}\Big) ( 1+\alpha) \mc D_{x-1}(f).
\end{equation}
The idea now is to use this estimate to transport $\mc D_{\ell}(f)$ to the boundary. If we use \eqref{taltal}  successively for $x=\ell,\dots,2$ we obtain the estimate
\[
\begin{split}
\mc D_{\ell}(f) 
	&\leq \frac{2(1+\beta)}{\eps_0} \sum_{x=1}^{\ell-1} \Big[\Big(1+\frac{1}{\beta}\Big)(1+ \alpha)\Big]^{\ell-1-x} \mc D_{x,x+1}(f) + \Big[\Big(1+\frac{1}{\beta}\Big)(1+ \alpha)\Big]^{\ell-1} \mc D_1(f)\\
	&\leq  \Big[\Big(1+\frac{1}{\beta}\Big)(1+ \alpha)\Big]^{\ell-1}\bigg( \frac{2(1+\beta)}{\eps_0} \sum_{x=1}^{\ell-1} \mc D_{x,x+1}(f) + \mc D_1(f)\bigg)\\
	&\leq \Big[\Big(1+\frac{1}{\beta}\Big)(1+ \alpha)\Big]^{\ell-1}
			\max\Big\{ \frac{2(1+\beta)}{\eps_0} , \frac{n}{\theta}\Big\} \mc D(f).
\end{split}
\]
Taking $\beta = n-1$ and using the bound $(1+a)^b \leq e^{ab}$ we obtain the estimate
\[
\mc D_{\ell}(f) \leq  n \max\Big\{ \frac{2}{\eps_0}, \frac{1}{\theta}\Big\} e^{\alpha n +1} \mc D(f).
\]
Since $\alpha n \leq  \tfrac{ \kappa}{\eps_0^2} $, the lemma is proved.
\end{proof}
 
 \begin{proof}[Proof of Theorem \ref{t6}]
We use the martingale method of Yau developed in \cite{LuYau}. We follow the approach underlined in \cite[Chp.~3]{yau}. Below, all conditional expectations are taken with respect to $\nu_{u(\cdot)}^n$. For every $u \geq 0$, define $\phi(u) := u \log u$. The key observation is that for every $f: \Omega_n \to \bb R$ and every $\sigma$-algebra $\mc G$,  if $g = E[f|\mc G]$ and $h =f/g$, then
\begin{equation}
\label{caldera}
H_{\nu_{u(\cdot)}^n}(f)  = \int \phi(f) d \nu_{u(\cdot)}^n = \int \phi(g) d \nu_{u(\cdot)}^n + \int \phi(h) g d \nu_{u(\cdot)}^n.
\end{equation}
Since $\nu_{u(\cdot)}^n$ is a product measure, we can use this relation to estimate the log-Sobolev constant $K_{\LS}$ recursively. For each $\ell \in \{2,\dots,n-1\}$, let 
\[
\mc G_\ell := \sigma(\eta(x); x \in \{1,\dots,\ell\}),
\]
the $\sigma$-algebra generated by the first $\ell$ coordinates of the configuration $\eta$.

Let us define
\[
\mc D^\ell(f) := \sum_{x=1}^{\ell-1} \mc D_{x,x+1}(f) + \frac{\theta}{n} \mc D_1(f)
\]
and
\[
K_\ell =K_\ell(\theta,\eps_0,\kappa) := \inf_{f,u} \frac{\mc D^\ell(\sqrt{f})}{\int \phi(f) d \nu_{u(\cdot)}^n},
\]
where the infimum runs over all densities $f$ with respect to $\nu_{u(\cdot)}^n$ which are measurable with respect to $\mc G_{\ell}$ and all profiles $u \in \mc U_{\eps_0,\kappa}^n$. Observe that $K_{\LS} = K_{n-1}(\eps_0,\kappa, \theta)$. 

Fix $u \in \mc U_{\eps_0,\kappa}^n$ and $\ell \in \{2,\dots,n-1\}$. Let $f$ be a density with respect to $\nu_{u(\cdot)}^n$ that is measurable with respect to $\mc G_\ell$. Define $g: \{0,1\} \to \bb R$ as $g(q) := E[f |\eta(\ell) =q]$ for each $q \in \{0,1\}$. Let $\nu_{u(\cdot)}^{\ell-1}$ be the law of $\xi := (\eta(x); x \in \{1,\dots,\ell-1\})$ and let $\widehat{\nu}_{u(\cdot)}^\ell$ be the law of $\eta(\ell)$ with respect to $\nu_{u(\cdot)}^n$. Observe that $f$ can be thought as a function of $(\xi,q)$.

By \eqref{caldera},
\[
\int \phi(f) d \nu_{u(\cdot)}^n = \int \phi(g) d \widehat{\nu}^\ell_{u(\cdot)} +
		\int \bigg( \int \phi\Big( \frac{f(\xi,q)}{g(q)} \Big) \nu_{u(\cdot)}^{\ell-1}(d \xi)\bigg) g(q) \widehat{\nu}^\ell_{u(\cdot)}(d q). 
\]
Observe that for each $q \in \{0,1\}$, the function $\xi \mapsto \frac{f(\xi,q)}{g(q)}$ is  $\mc G_{\ell-1}$-measurable and it is a density with respect to $\nu_{u(\cdot)}^n$. By the definition of $K_\ell$,
\[
 \int \phi\Big( \frac{f(\xi,q)}{g(q)} \Big) \nu_{u(\cdot)}^{\ell-1}(d \xi) \leq \frac{1}{K_{\ell-1} g(q)} \mc D^{\ell-1}(\sqrt{f(\cdot,q)}).
 \]
 Therefore, we have that
 \[
 \int \phi(f)  d \nu_{u(\cdot)}^n \leq \int \phi(g)  \widehat{\nu}^\ell_{u(\cdot)}(d \theta) + \frac{1}{K_{\ell-1}} \mc D^{\ell-1}(f).
 \]
 Notice that $ \widehat{\nu}^\ell_{u(\cdot)}$ is a Bernoulli law of parameter $u(\ell)$. By \cite[Lemma 2]{LeeYau}, there exists a constant $B$ independent of $\xi$ and $u(\ell)$ such that
 \[
  \int \phi(g)  \widehat{\nu}^\ell_{u(\cdot)}(d q) \leq  B \Big( \sqrt{g(1)} - \sqrt{g(0)}\Big)^2.
 \]
 Let $X$ and $Y$ be non-negative random variables. Observe that 
 \[
 \big( \sqrt{\bb E [X]} - \sqrt{\bb E [Y]} \big)^2 \leq \bb E \big[ \big( \sqrt{X} - \sqrt{Y}\big)^2 \big].
 \]
 Taking $X(\xi) = f(\xi, 1)$ and $Y(\xi) = f(\xi,0)$, we see that
 \[
 \big( \sqrt{g(1)} - \sqrt{g(0)} \big)^2 \leq \int \big( \sqrt{f(\xi,1)} - \sqrt{f(\xi,0)}\big)^2 \nu_{u(\cdot)}^{\ell-1}(d \xi) = \frac{\mc D_{\ell}(\sqrt{f})}{2 u(\ell)(1-u(\ell))} \leq \frac{\mc D_{\ell} (\sqrt{f})}{\eps_0}.
 \]
 Using Lemma \ref{l1} with $n =\ell+1$, we see that
\[
\begin{split}
\int \phi(f)  d \nu_{u(\cdot)}^n 
	&\leq \frac{B}{\eps_0} \mc D_\ell (\sqrt{f}) + \frac{1}{K_{\ell-1}} \mc D^{\ell-1}(\sqrt f)\\
	&\leq \Big(  \frac{B C n}{\eps_0} + \frac{1}{K_{\ell-1}} \Big) \mc D^\ell(\sqrt{f})
\end{split}
\] 
from where 
\[
\frac{1}{K_\ell} \leq \frac{BC n}{\eps_0} + \frac{1}{K_{\ell-1}}
\]
for some finite constant $C = C(\eps_0,\kappa, \theta)$. 
Therefore, $K_\ell^{-1} \leq \frac{B C n \ell}{\eps_0}+K_{2}^{-1}$. Observe that $K_2^{-1} \leq B$. Taking $\ell =n-1$, the proof is complete.
\end{proof}
 
 \begin{proof}[Proof of Theorem \ref{t5}]
 First we observe that for any $f: \Omega_n \to \bb R$, 
 \[
\int  \Gamma_n f d \nu_{u(\cdot)}^n = n^2 \sum_{x=1}^{n-2} \mc D_{x,x+1}(f) + n^2 \!\!\!\!\!\!\sum_{x \in \{1,n-1\}}\int \big(\rho (1-\eta(x))+(1-\rho) \eta(x)\big) \big(\nabla_x f(\eta)\big)^2 d \nu_{u(\cdot)}^n.
 \]
 Therefore, for $\theta = n \min\{\rho,1-\rho\}$,
 \[
 n^2 \mc D(f) \leq \int \Gamma_n f d \nu_{u(\cdot)}^n
 \]
 for every $n \in \{2,3,\dots\}$ and every $f: \Omega_n \to \bb R$. Theorem \ref{t5} follows from this bounds and Theorem \ref{t6}, with $K_0(\rho,\eps_0,\kappa) = K(\eps_0,\kappa, \min\{\rho,1-\rho\})$.
 \end{proof}

\section{The relative entropy method}
\label{s5}
In this section we prove Theorem \ref{t2}. The proof follows the so-called \emph{Yau's relative entropy method}, introduced in \cite{yaurem}.
Let us recall the definition of the carr\'e du champ $\Gamma_n$: for every $f: \Omega_n \to \bb R$,  $\Gamma_n f = \mf{L}_n f^2 - 2 f \mf{L}_n f$. 
Let us use the product Bernoulli measure $\bar{\nu}_{\hspace{-1.5pt}\rho}^{\hspace{0.3pt}n}$ as a reference measure in $\Omega_n$. Recall the reference measures $\nu_t^n = \nu_{u_t^n(\cdot)}$ defined after \eqref{antofagasta}. Let $\psi_t^n: \Omega_n \to [0,\infty)$ be the Radon-Nikodym derivative of $\nu_t^n$ with respect to $\bar{\nu}_{\hspace{-1.5pt}\rho}^{\hspace{0.3pt}n}$, that is, $\psi_t^n(\eta) = \frac{\nu_t^n(\eta)}{\bar{\nu}_{\hspace{-1.5pt}\rho}^{\hspace{0.3pt}n}(\eta)}$ for every $\eta \in \Omega_n$. Let $\mf L^\ast_{n,t}$ be the adjoint of $\mf L_n$ with respect to $\nu_t^n$. The action of $\mf L_{n,t}^\ast$ over a function $g: \Omega_n \to \bb R$ is given by
\begin{equation}
\label{copiapo}
\begin{split}
\mf L_{n,t}^\ast g (\eta) 
	&:= n^2 \sum_{x=1}^{n-2} \Big( g(\eta^{x,x+1}) \frac{\nu_t^n(\eta^{x,x+1})}{\nu_t^n(\eta)} - g(\eta)\Big)\\
	& + n^2 \!\!\!\!\!\! \sum_{x \in \{1,n-1\}} \!\!\!\! \Big(
	\big( \rho \eta(x) + (1-\rho)(1-\eta(x)\big) g(\eta^x) \frac{\nu_t^n(\eta^x)}{\nu_t^n(\eta)} -
	\big( \rho(1-\eta(x)) + (1-\rho) \eta(x) \big) g(\eta) \Big)
\end{split}
\end{equation}
for every $\eta \in \Omega_n$. 

Yau's relative entropy inequality asserts the following (see \cite{mainlemma}):

\begin{proposition}[Yau's inequality]
\label{p2}
For each $t \geq 0$, let $\mu_t^n$ be the law of $\eta_t$ with respect to $\bb P_{\nu_0^n}$ and let $f_t^n$ be the Radon-Nikodym derivative of $\mu_t^n$ with respect to $\nu_t^n$. Recall that $H_n(t) := H_{\nu_t^n}(f_t^n)$. We have that 
\[
\tfrac{d}{dt} H_n(t) \leq - \int \Gamma_n \sqrt{f_t^n} d \nu_t^n + \int f_t^n \big( \mf L_{n,t}^\ast \mathds 1 - \partial_t \log \psi_t^n \big) d \nu_t^n,
\]
where $\mathds 1$ is the constant function equal to $1$.
\end{proposition}

Observe that
\begin{equation}
\label{paipote}
\psi_t^n(\eta) = \prod_{x \in \Lambda_n} \bigg( \eta(x) \frac{u_t^n(x)}{\rho} + (1-\eta(x)) \frac{1-u_t^n(x)}{1-\rho} \bigg).
\end{equation}
Using this expression and \eqref{copiapo}, it is possible to compute $\mf L_{n,t}^\ast \mathds 1 - \partial_t \log \psi_t^n$ in an explicit way. For each $x \in \Lambda_n$, let us define
\[
\omega_x := \frac{\eta(x)- u_t^n(x)}{u_t^n(x) (1-u_t^n(x))}.
\]
Observe that $\omega_x$ also depends on $t$. In order to make the notation more compact, we will not include this dependence on the notation. We have that
\[
\mf L_{n,t}^\ast \mathds 1 - \partial_t \log \psi_t^n
		= 
		 -\sum_{x=1}^{n-2} n^2\big( u_t^n(x+1) -u_t^n(x)\big)^2 \omega_x \omega_{x+1}.
\]
Observe that to integrate with respect to $f_t^n d \nu_t^n$ is equivalent to take expectations with respect to $\bb E_{\nu_0^n}$. Therefore,
\begin{equation}
\label{huasco}
\tfrac{d}{dt} H_n(t) \leq - \int \Gamma_n \sqrt{f_t^n} d \nu_t^n -  \sum_{x=1}^{n-2} n^2\big( u_t^n(x+1) -u_t^n(x)\big)^2 \bb E_{\nu^n_0}\big[  \omega_x \omega_{x+1} \big].
\end{equation}
We see that it would be good to have an estimate for $\bb E_{\nu_0^n}[\omega_x \omega_{x+1}]$. The following proposition follows from \cite[Lemma 4.1 and Proposition 4.4]{lmo}:

\begin{proposition}
\label{plmo}
For every $n \in \{2,3,\dots\}$, every profile $u_0 \in \mc U_{\eps_0,\kappa}$, every $x \in \Lambda_n$ and every $t \geq 0$,
\[
\big|\bb E_{\nu_0^n}[\omega_x \omega_{x+1}] \big| \leq \frac{\kappa^2}{\eps_0^2 n}
\]
and
\begin{equation}
\label{freirina}
n \big| u_t^n(x+1)- u_t^n(x) \big| \leq \kappa,
\end{equation}
where $\{u_t^n \scn t \geq 0\}$ is the solution of  \eqref{antofagasta}.
\end{proposition}
Putting these estimates into \eqref{huasco}, we see that
\begin{equation}
\label{vallenar}
\tfrac{d}{dt} H_n(t) \leq - \int \Gamma_n \sqrt{f_t^n} d \nu_t^n +  \frac{\kappa^4}{\eps_0^2} .
\end{equation}
By Theorem \ref{t5},
\[
\tfrac{d}{dt} H_n(t) \leq - K_0 H_n(t) + \frac{\kappa^4}{\eps_0^2}.
\]
Using the integrating factor $e^{K_0 t}$, this exponential inequality can be integrated out, from where we conclude that
\begin{equation}
\label{domeyko}
H_n(t) \leq \frac{\kappa^4}{K_0 \eps_0^2}
\end{equation}
for every $t \geq 0$, so
$H_n(t)$ is uniformly bounded in $t$ by a constant that only depends on $\rho$, $\eps_0$ and $\kappa$. 

In order to show that $H_n(t)$ decays to $0$ in $t$, we need to take advantage of the presence of the discrete gradient $n  (u_t^n(x+1)-u_t^n(x))$ in the expression for $\mf L_{n,t}^\ast \mathds 1 - \partial_t \log \psi_t^n$. By  Lemma \ref{la1},
\[
n\big| u_t^n(x+1)-u_t^n(x) \big| \leq 8 \pi e^{- \lambda_1^n t} 
\]
for every $n \in \{2,3,\dots\}$, every $x \in \Lambda_n$ and every $t \geq \tfrac{1}{\lambda_1^n} \log 2$, where $\lambda_1^n := 4 n^2 \sin^2\big(\tfrac{\pi}{2 n} \big)$. Therefore, for $t \geq t_0^n := \tfrac{1}{ \lambda_1^n} \log 2$,
\[
\tfrac{d}{dt} H_n(t) \leq -K_0 H_n(t) + \frac{2^6 \pi^2 \kappa^2 e^{-2 \lambda_1^n t}}{\eps_0^2}.
\]
Assume that $2 \lambda_1^n > K_0$. Integrating between $t_0^n$ and $t_0^n+t$ and using \eqref{domeyko} to estimate $H_n(t_0)$, we conclude that
\[
H_n(t_0^n+t) \leq \Big(\frac{\kappa^4}{K_0 \eps_0^2} +\frac{2^6 \pi^2 \kappa^2}{(2 \lambda_1^n - K_0)\eps_0^2} \Big)e^{-K_0 t}
\]
If $2 \lambda_1^n < K_0$, the estimate holds with exponential factor $e^{ -2  \lambda_1^n t}$. If $2  \lambda_1^n =K_0$, the estimate holds with exponential factor $e^{-(2  \lambda_1^n -\delta) t}$ for any $\delta >0$. By Lemma \ref{la0.1}, $\lambda_1^n \geq \tfrac{3 \pi^2}{4}$. Taking $t_0 = \tfrac{4}{3 \pi^2 } \log 2$, in each case Theorem \ref{t2} is proved.

\section{Total variation distance between profile measures}
\label{s6}

In this section we prove Theorem \ref{t3}. By definition,
\[
\| \nu_t^n - \bar{\nu}_{\hspace{-1.5pt}\rho}^{\hspace{0.3pt}n} \|_{\TV} = \frac{1}{2}\int | \psi_t^n -1| d \bar{\nu}_{\hspace{-1.5pt}\rho}^{\hspace{0.3pt}n}.
\]
Recall that $\psi_t^n$ has an explicit formula, see \eqref{paipote}. Therefore, the analysis is reduced to an asymptotic analysis of $\psi_t$. Observe that $\psi_t^n$ can be written in the form
\[
\psi_t^n := \exp \Big( \sum_{x \in \Lambda_n} \big( a_t^n(x) (\eta_x-\rho) - b_t^n(x)\big) \Big)
\]
for 
\[
a_t^n(x) = \log \tfrac{u_t^n(x)}{\rho} - \log \tfrac{1-u_t^n(x)}{1-\rho} 
\]
 and 
 \[
 b_t^n(x) = -\rho \log \tfrac{u_t^n(x)}{\rho} - (1-\rho) \log \tfrac{1-u_t^n(x)}{1-\rho}.
 \]
 The sum 
\[
\sum_{x \in \Lambda_n}  a_t^n(x) (\eta_x-\rho)
\]
can be understood as a triangular array of independent, centred random variables, and in particular it should converge, after a suitable renormalisation, to a Gaussian random variable. One can verify that the computations we perform below allow us to prove that the Lyapounov condition is satisfied for this sum, and therefore the convergence to a Gaussian random variable can be justified. These considerations suggest to define
\[
s_t^n : = \Big( \rho(1-\rho) \sum_{x \in \Lambda_n} a_t^n(x)^2 \Big)^{1/2},
\]
\[
X_t^n := \frac{1}{s_t^n} \sum_{x \in \Lambda_n} a_t^n(x) (\eta_x - \rho),
\]
\[
b_t^n:=  \sum_{x \in \Lambda_n} b_t^n(x).
\]
With these notations, we have that
\[
\psi_t^n = \exp\{ s_t^n X_t^n - b_t^n\},
\]
and the analysis of $\| \nu_t^n - \bar{\nu}_{\hspace{-1.5pt}\rho}^{\hspace{0.3pt}n} \|_{\TV}$ reduces to the analysis of $s_t^n$, $X_t^n$ and $b_t^n$. Recall that we are interested in the behavior of these quantities for $t = t^n(b)$ as defined in Theorem \ref{t1}. From now on we fix $B >0$ and take $b \in [-B,B]$. Hereafter we denote by $R_t^{n,i}(x)$ an error term that goes to $0$ as $n \to \infty$, uniformly in $x \in \Lambda_n$ and $b \in [-B,B]$. The index $i$ serves to indicate the places at which the error term changes. By Lemma \ref{la8}, 
\[
u_t^n(x) = \rho + \tfrac{1}{\sqrt n} \big( c_{\ell_0}(u_0) e^{-b}\phi_{\ell_0}^n(x) +R_t^{n,1}(x)\big),
\]
where $\phi_{\ell_0}^n(x) := \sqrt{2}\sin \big( \tfrac{\pi \ell_0 x}{n} \big)$.
By Taylor's formula, $\log (1+x) = x -x^2/2+\mathcal O(x^3)$. Therefore, 
\begin{equation}
\label{lebu}
a_t^n(x) = \frac{1}{\sqrt n} \bigg( \frac{c_{\ell_0}(u_0)e^{-b}\phi_{\ell_0}^n(x)}{\rho(1-\rho)} + R_t^{n,2}(x)\bigg).
\end{equation}
Now we can compute $s_t^n$:
\[
(s_t^n)^2 = \rho(1-\rho) \sum_{x \in \Lambda_n} a_t^n(x)^2 = \frac{\rho(1-\rho)}{n} \sum_{x \in \Lambda_n} \bigg(\frac{c_{\ell_0}(u_0)^2 e^{-2b} \phi_{\ell_0}^n(x)^2}{\rho^2(1-\rho)^2}+ R_t^{n,3}(x)\bigg).
\]
Since 
\[
\frac{1}{n} \sum_{x \in \Lambda_n} \phi_{\ell_0}^n(x)^2
\]
is a Riemann sum of the integral $2 \int_0^1 \sin^2(\pi \ell_0 x) dx$, which is equal to $1$, 
we see that
\[
(s_t^n)^2 = \frac{c_{\ell_0}(u_0)^2 e^{-2b}}{ \rho(1-\rho)} + R_t^{n,4},
\]
from where
\[
\lim_{n \to \infty} \sup_{b \in [-B,B]} \Big| s_t^n - \frac{\big|c_{\ell_0}(u_0)\big| e^{-b}}{\sqrt{ \rho(1-\rho)}} \Big| =0.
\]
In order to compute $b_t^n$, we need to go one step further in the Taylor's expansion of $\log(1+x)$: 
\[
\log (1+x) = x - \tfrac{1}{2} x^2 +\frac{x^3}{3} +\mathcal O (x^4).
\]
Proceeding as before we see that
\[
b_t^n = \frac{1}{n} \sum_{x \in \Lambda_n} \bigg(\frac{c_{\ell_0}(u_0)^2e^{-2b} \phi_\ell^n(x)^2}{2 \rho(1-\rho)} + R_t^{n,5}(x)\bigg)
	= \frac{c_{\ell_0}(u_0)^2 e^{-2b}}{2 \rho(1-\rho)} + R_t^{n,6},
\]
and in particular we see that $b_t^n$ and $\tfrac{1}{2} (s_t^n)^2$ have the same limit as $n \to \infty$.

Recall that we want to obtain the limit as $n \to \infty$ of
\begin{equation}
\label{lota}
\tfrac{1}{2} \int \big| \exp \big\{ s_t^n X_t^n - b_t^n \big\} -1 \big| d \bar{\nu}_{\hspace{-1.5pt}\rho}^{\hspace{0.3pt}n}.
\end{equation}
Up to here, we have proved the convergence of $s_t^n$ and $b_t^n$. By Lyapounov's criterion with fourth moment condition, $X_t^n$ converges in law to a standard Gaussian law if
\[
\lim_{n \to \infty} \frac{1}{(s_t^n)^4}\int \sum_{x \in \Lambda_n} a_t^n(x)^4 (\eta_x -\rho)^4 d \bar{\nu}_{\hspace{-1.5pt}\rho}^{\hspace{0.3pt}n} =0.
\]
Observe that $\int (\eta_x -\rho)^4 d \bar{\nu}_{\hspace{-1.5pt}\rho}^{\hspace{0.3pt}n}= \rho(1-\rho)(1-3\rho + 3 \rho^2)$. The actual value of this integral is not relevant; it is only relevant that it is constant in $n$, $x$ and $t$. Since $s_t^n$ has a non-zero limit, we only need to prove that
\[
\lim_{n \to \infty} \sum_{x \in \Lambda_n} a_t^n(x)^4 = 0.
\]
From \eqref{lebu}, we see that there exists a finite constant $C = C(u_0, \ell_0, B)$ such that $|a_t^n(x)| \leq \frac{C}{\sqrt n}$ for every $n \in \{2,3,\dots\}$, every $x \in \Lambda_n$ and every $b \in [-B,B]$. Therefore,
\[
\sum_{x \in \Lambda_n} a_t^n(x)^4 \leq \frac{C^4}{n}
\]
and Lyapounov's condition is satisfied.

Up to here we have proved that $s_t^n X_t^n -b_t^n$ converges in law to $\gamma e^{-b} X - \tfrac{1}{2} \gamma^2 e^{-2b}$, where $\gamma = \frac{|c_{\ell_0}(u_0)|}{\sqrt{ \rho(1-\rho)}}$  and $X$ has a standard Gaussian law. Since the exponential function is not bounded, one needs an additional argument in order to prove that \eqref{lota} converges. The integral \eqref{lota} converges to
\[
\tfrac{1}{2} \bb E\big[ \big| e^{\gamma e^{-b} X - \tfrac{1}{2} \gamma^2 e^{-2b}} -1 \big| \big]
\]
if the sequence $\{e^{s_t^n X_t^n}; n \in \{2,3,\dots\}\}$ is uniformly integrable. Since $L^p$-boundedness for some $p>1$ implies uniform integrability , it is enough to show that $\int e^{p s_t^n X_t^n} d \bar{\nu}_{\hspace{-1.5pt}\rho}^{\hspace{0.3pt}n}$ is uniformly bounded for at least one $p >1$. From Hoeffding's inequality,
\[
\int e^{p s_t^n X_t^n} d \bar{\nu}_{\hspace{-1.5pt}\rho}^{\hspace{0.3pt}n} \leq \exp\Big\{ \sum_{x \in \Lambda_n} \tfrac{1}{8} p^2 a_t^n(x)^2\Big\} \leq \exp \big\{ \tfrac{1}{8} C^2 p^2 \big\}.
\]
We conclude that for any $b \in \bb R$,
\[
\lim_{n \to \infty} \tfrac{1}{2} \int | \psi_t^n -1| d \bar{\nu}_{\hspace{-1.5pt}\rho}^{\hspace{0.3pt}n} = \tfrac{1}{2} \bb E \big[ \big| e^{\gamma e^{-b} X -\tfrac{1}{2} \gamma^2 e^{-2b} } -1 \big| \big] = \mc G( \gamma e^{-b}).
\]
This shows that
\[
\lim_{n \to \infty} \| \nu^n_{t^n(b)} - \bar{\nu}_{\hspace{-1.5pt}\rho}^{\hspace{0.3pt}n} \|_{\TV}
		= \mc G(\gamma e^{-b}).
\]

\appendix

\section{Estimates on the discrete heat equation}

In this section we collect and prove various facts about solutions of \eqref{antofagasta}. Using discrete Fourier transform, \eqref{antofagasta} can be solved in terms of trigonometric functions: for $n \in \{2, 3,\dots\}$ and $\ell \in \{1,\dots,n-1\}$, define $\phi_\ell^n: \Lambda_n \to \bb R$ as
\[
\phi_\ell^n(x) := \sqrt{2} \sin \big( \tfrac{\pi \ell x}{n}\big)
\]
for every $x \in \Lambda_n$ and define
\[
\lambda_\ell^n := 2 n^2 \big(1- \cos \big( \tfrac{\pi \ell}{n}\big)\big) = 4 n^2 \sin^2\big( \tfrac{\pi \ell}{2 n}\big).
\]
Observe that $\Delta_n \phi_\ell^n(x) = - \lambda_\ell^n \phi_\ell^n(x)$ for every $x \in \Lambda_n$. The solution $\{u_t^n(x); x \in \Lambda_n, t \geq 0\}$ has the following representation:
\begin{equation}
\label{renaico}
u_t^n(x) = \rho + \sum_{\ell=1}^{n-1} c_\ell^n e^{-\lambda_\ell^n t}  \phi_\ell^n(x)
\end{equation}
for every $x \in \Lambda_n$ and every $t \geq 0$, where the numbers
\[
c_\ell^n = c_\ell^n(u_0):= \frac{1}{n} \sum_{x \in \Lambda_n}( u_0(x)-\rho) \phi_\ell^n(x), \ell \in \{1,\dots,n-1\}
\]
are the \emph{Fourier coefficients} of $u_0$. 

Our first lemma gives a very useful estimate for the eigenvalues $\lambda_\ell^n$:

\begin{lemma}
\label{la0}
For each $n \in \{2,3,\cdots\}$ and each $\ell_0, \ell \in \{1,\dots n-1\}$ such that $\ell_0 \leq \min\{\ell,\frac{n}{2}\}$,
\[
\frac{\lambda_\ell^n}{\lambda_{\ell_0}^n} \geq \frac{ \ell}{\ell_0}.
\]
\end{lemma}

\begin{proof}
As $n \to \infty$, $\lambda_{\ell}^n \sim \pi^2 \ell^2$ as long as $\ell = o(\sqrt{n})$, so in fact the ratio in the statement of the lemma approaches $\frac{\ell^2}{\ell_0^2}$. The point on this lemma is that the estimate is uniform in $\ell_0$, $\ell$ and $n$. Consider $f(x) := 1 - \cos x$ and fix $x_0 \in (0,\frac{\pi}{2}]$. Then, $f'(x) = \sin x \geq \sin x_0$ for every $x \in [x_0,\pi -x_0]$. Integrating this inequality in $x$, we see that
\[
f(x) \geq f(x_0) + (x-x_0) \sin x_0,
\]
from where
\[
\frac{f(x)}{f(x_0)} \geq 1 +\frac{\sin x_0}{1-\cos x_0} (x-x_0) = \frac{x}{x_0} + (x-x_0) \Big( \frac{\sin x_0}{1-\cos x_0} - \frac{1}{x_0}\Big)
\]
for every $x_0 \in (0, \tfrac{\pi}{2}]$ and every $x \in [x_0,\pi - x_0]$.
Therefore, the lemma is proved if we show that
\begin{equation}
\label{mulchen}
\frac{\sin x_0}{1-\cos x_0} - \frac{1}{x_0} \geq 0,
\end{equation}
since we can take $x_0 = \tfrac{\pi \ell_0}{n}$ and $x = \tfrac{\pi \ell}{n}$.
Observe that
\[
\frac{\sin x_0}{1-\cos x_0} -\frac{1}{x_0} =  \cot \big( \tfrac{x_0}{2} \big).
\]
Therefore, the difference in \eqref{mulchen} is asymptotically equivalent to $\tfrac{1}{x_0}$ for $x_0 \ll 1$ and it is decreasing in $x_0$. For $x_0=\tfrac{\pi}{2}$, the difference in \eqref{mulchen} is equal to $1- \tfrac{2}{\pi} >0$, so the lemma is proved.
\end{proof}

The following lemma is useful whenever we need a rough estimate of the right order in $\lambda_\ell^n$:

\begin{lemma}
\label{la0.1}
For every $x \geq 0$,
\[
\tfrac{1}{2} x^2 \geq 1- \cos x \geq \tfrac{1}{2} x^2 \big( 1- \tfrac{1}{12} x^2 \big).
\] 
In particular, for every $n \in \{2,3,\dots\}$, 
\[
\lambda_1^n \geq \pi^2\big( 1 -\tfrac{\pi^2}{12 n^2}\big) \geq \pi^2 \big(1 -\tfrac{\pi^2}{48} \big) \geq \tfrac{3 \pi^2}{4}
\] 
and for every $\ell \in \{1,\dots,n-1\}$,
\[
\Big|\frac{\lambda_\ell^n}{ \pi^2 \ell^2} -1 \Big| \leq \tfrac{\pi^2 \ell^2}{12 n^2}.
\]
\end{lemma}

Our next estimate establishes the exponential decay of the $\ell_\infty$-norm of the gradient of $u_t^n$:

\begin{lemma}
\label{la1}
For every $n \in \{2,3,\dots\}$, every $ u_0: \Lambda_n \to [0,1]$, every $x \in \Lambda_n$ and every $t \geq \tfrac{1}{\lambda_1^n} \log 2$, the solution of \eqref{antofagasta} satisfies
\begin{equation}
\label{angol}
n\big| u_t^n(x+1) - u_t^n(x)\big| \leq 8 \pi e^{- \lambda_1^n t}.
\end{equation}
\end{lemma}
\begin{proof}
Since $u_0 \in [0,1]$, $|c_\ell^n| \leq 2$ for every $\ell \in \{1,\dots,n-1\}$. Using the bound
\[
|\sin(x) - \sin(y) | \leq |x-y|,
\]
valid for every $x, y \in \bb R$, we see that 
\begin{equation}
\label{temuco}
n | u_t^n(x+1) -u_t^n(x)| \leq \sum_{\ell=1}^{n-1} 2 n e^{- \lambda_\ell^n t} \big| \sin \big( \tfrac{\pi \ell x}{n}\big) - \sin \big( \tfrac{\pi \ell(x+1)}{n}\big) \big|  \leq \sum_{\ell=1}^{n-1} 2 \pi \ell e^{-\lambda_\ell^n t}.
\end{equation}
From Lemma \ref{la0},
\[
\sum_{\ell=1}^{n-1} 2 \pi \ell e^{-\lambda_\ell^n t} \leq \sum_{\ell=1}^\infty 2 \pi \ell e^{- \lambda_1^n \ell t} = 
	\frac{2 \pi e^{- \lambda_1^n t}}{(1-e^{- \lambda_1^n t})^2}.
\]
Putting this estimate into \eqref{temuco}, we obtain the estimate
\begin{equation}
\label{tolten}
n | u_t^n(x+1) -u_t^n(x)| \leq \sum_{\ell=1}^{n-1} 2 \pi \ell e^{-\lambda_1^n \ell t} \leq \frac{2 \pi e^{-\lambda_1^n t}}{(1-e^{-\lambda_1^n t})^2}.
\end{equation}
For every $t \geq \frac{1}{\lambda_1^n} \log 2$, the denominator of this expression is bounded below by $\tfrac{1}{4}$, which proves the lemma.
\end{proof}

\begin{remark}
Observe that in this lemma we are not assuming any condition on the Lipschitz constant of $u_0$. Therefore, a lower bound on the times $t$ at which \eqref{angol} holds is needed. In particular the restriction $t \geq \tfrac{1}{\lambda_1^n} \log 2$ is sharp up to a constant.
\end{remark}

\begin{remark}
With more careful computations, it is possible to replace $\lambda_1^n$ by $\pi^2$, at the cost of taking $n$ large enough and taking $t  \leq n^2$. Since we only need an exponential decay in this lemma, we did not pursue this more refined bound.
\end{remark}

Define $\phi_\ell: [0,1] \to \bb R$ as $\phi_\ell(x) = \sqrt{2} \sin ( \pi \ell x)$ for every $x \in [0,1]$. Observe that $\phi_\ell^n(x) = \phi_\ell \big(\tfrac{x}{n} \big)$. Observe as well that the Fourier coefficients $c_\ell^n$ are Riemann sums of the integrals
\[
c_\ell(u_0) :=\int ( u_0(x)-\rho) \phi_\ell(x) dx,
\]
which are the Fourier coefficients of $u_0$ in the continuous interval. We have the following lemma:

\begin{lemma}
\label{la6}
Let $\kappa>0$ and $\ell _0 \in \{2,3.\dots\}$. There exists a constant $C = C(\kappa,\ell_0)$ such that
\[
\big| c_\ell^n(u_0) - c_\ell(u_0) \big| \leq \frac{C}{n}
\]
for every $n \in  \{\ell_0+1,\ell_0+2,\dots\}$, every $\ell \in \{1,\dots,\ell_0\}$ and every $u_0: [0,1] \to [0,1]$ such that $u_0(0)=u_0(1)=\rho$ and $\|u'\|_\infty \leq \kappa$.
\end{lemma}

\begin{proof}
It is well known that for any differentiable function $f:[0,1] \to \bb R$ satisfying $f(0) = f(1) =0$,
\begin{equation}
\label{Riemann}
\bigg| \frac{1}{n} \sum_{x \in \Lambda_n} f\big( \tfrac{x}{n} \big) - \int f(x) dx \bigg| \leq \frac{\|f'\|_\infty}{2 n}.
\end{equation}
The lemma follows by computing the derivatives of the functions $(u_0(x)-\rho) \phi_\ell(x)$ for $\ell \leq \ell_0$.
\end{proof}
\begin{remark}
The error in \eqref{Riemann}
is bounded by $\tfrac{\|f''\|}{24 n^2}$ if $f$ is twice differentiable, but we do not need that enhanced precision here. 
\end{remark}

Our next lemma derives the asymptotic behaviour of $u_t^n$ on the relevant time window:

\begin{lemma}
\label{la8}
Let $u_0: [0,1] \to [0,1]$ be differentiable, such that $u(0)=u(1)=\rho$. Let $\ell_0 \in \bb N$ be the smallest integer such that $c_{\ell_0}(u_0) \neq 0$. For every $B >0$, 
\[
\lim_{n \to \infty} \sup_{|b|\leq B} \sup_{x \in \Lambda_n}  \big| \sqrt{n} (u_{t^n(b)}^n(x) - \rho ) - c_{\ell_0}(u_0) e^{-b} \phi_{\ell_0}^n(x)\big| =0,
\]
where $t^n(b)$ is defined in Theorem \ref{t1}.
\end{lemma}

\begin{proof}
The idea is to divide the sum in \eqref{renaico} into two parts:
\[
\sqrt{n} \big| u_t^n(x) - \rho -  c_{\ell_0}(u_0) e^{- \lambda_{\ell_0}^n t} \phi_{\ell_0}^n(x) \big| 
		\leq \sqrt{2n}\sum_{\ell=1}^{\ell_0} \big|  c_\ell^n(u_0) -c_\ell (u_0)\big| 
				+ \sqrt{n}\sum_{\ell=\ell_0+1}^n 2 e^{-\lambda_\ell^n t }.
\]
Later on we will take $t = t^n(b)$.
By Lemma \ref{la6},
\[
\sqrt{2n} \sum_{\ell=1}^{\ell_0} \big|  c_\ell^n(u_0) -c_\ell (u_0)\big|  \leq \frac{C(\|u'\|_\infty, \ell_0)}{\sqrt n}.
\]
By Lemma \ref{la0},
\[
\sqrt{n} \sum_{\ell=\ell_0+1}^n 2 e^{-\lambda_\ell^n t } 
 		\leq \sqrt{n}\sum_{\ell = \ell_0+1}^\infty 2 e^{-\frac{\lambda_{\ell_0+1}^n }{\ell_0+1} \ell t} 
			= \frac{2 \sqrt{n} e^{-\lambda_{\ell_0+1}^n t}}{1-e^{-\lambda_{\ell_0+1}^n t}}. 
\]
Observe that 
\[
\lim_{n \to \infty} \frac{\lambda_{\ell_0+1}^n}{\pi^2 \ell_0^2} = \big( 1+ \tfrac{1}{\ell_0}\big)^2. 
\]
Therefore, there exists $n_0 = n_0(\ell_0)$ such that $\frac{\lambda_{\ell_0+1}^n}{\pi^2 \ell_0^2} \geq 1 + \frac{1}{\ell_0}$ for every $n \geq n_0$. Observe that the function $s \mapsto \frac{e^{-s}}{1-e^{-s}}$ is decreasing in $s$. Therefore, for every $b \in [-B,B]$,
\[
 \sqrt{n}\sum_{\ell=\ell_0+1}^n 2 e^{-\lambda_\ell^n t^n(b) } \leq \frac{2 \sqrt{n} e^{-( 1 +\frac{1}{\ell_0})(\frac{1}{2} \log n - B)}}{1-e^{-( 1 +\frac{1}{\ell_0})(\frac{1}{2} \log n - B)}} \leq \frac{2e^{2B}}{n^{\frac{1}{2 \ell_0}}(1-e^{-2B})} \leq \frac{C(B)}{n^{\frac{1}{2 \ell_0}}}.
 \]
 The numerical value of the constant $C(B)$ is not really important, the decay in $n$ is what we need. Observing that 
 \[
e^{- b - \tfrac{B \pi^2 \ell_0^2}{n^2}} \exp\big\{ - b - \tfrac{B \pi^2 \ell_0^2}{n^2} \big\}\leq \sqrt{n} e^{-\lambda_{\ell_0}^n t^n(b)} \leq e^{-b + \tfrac{\pi^2 \ell_0^2}{n^2} ( \tfrac{1}{2} \log n +B) },
 \]
 the lemma is proved.
\end{proof}

\bibliographystyle{plain}

\end{document}